\documentclass[10pt]{article}
\usepackage{amsfonts,amssymb,amsmath,comment}
\usepackage{graphicx,graphics,color}
\usepackage[english]{babel}
\usepackage{tikz}

 \makeatletter
 \@addtoreset{equation}{section}
 \makeatother

 \newcounter{enunciato}[section]

 \newtheorem{ittheorem}{Theorem}
 \newtheorem{itlemma}{Lemma}
 \newtheorem{itproposition}{Proposition}
 \newtheorem{itdefinition}{Definition}
 \newtheorem{itcorollary}{Corollary}
 \newtheorem{itconjecture}{Conjecture}

 \newenvironment{theorem}{\addtocounter{enunciato}{1}
 \begin{ittheorem}}{\end{ittheorem}}

 \newenvironment{lemma}{\addtocounter{enunciato}{1}
 \begin{itlemma}}{\end{itlemma}}

\newenvironment{proof}{\noindent {\em Proof}.\,\,}
{\hspace*{\fill }$\square $}

\def \ba {\begin{array}}
\def \ea {\end{array}}

\def \R {{\mathbb R}}

\begin{document}
\title{Heat flow from polygons}

\author{\renewcommand{\thefootnote}{\arabic{footnote}}
M.\ van den Berg
\footnotemark[1]
\\
\renewcommand{\thefootnote}{\arabic{footnote}}
P.\ B.\ Gilkey
\footnotemark[2]
\\
\renewcommand{\thefootnote}{\arabic{footnote}}
K.\ Gittins
\footnotemark[3]
}

\footnotetext[1]{School of Mathematics, University of Bristol, Fry Building, Woodland Road, Bristol BS8 1UG, United Kingdom,\, {\tt mamvdb@bristol.ac.uk} }

\footnotetext[2]{Mathematics Department, University of Oregon,
Eugene, OR 97403, USA,\, {\tt gilkey@uoregon.edu}}

\footnotetext[3]{Universit\'e de Neuch\^atel, Institut de Math\'ematiques, Rue Emile-Argand 11, CH-2000
Neuch\^atel, Switzerland,\, {\tt katie.gittins@unine.ch}}

\date{{22 July 2019}}

\maketitle

\begin{abstract}
We study the heat flow from an open, bounded set $D$ in $\R^2$ with a polygonal boundary $\partial D$.
The initial condition is the indicator function of $D$.
A Dirichlet $0$ boundary condition has been imposed on some but not all of the edges of $\partial D$.
We calculate the heat content of $D$ in $\R^2$ at $t$
up to an exponentially small remainder as $t\downarrow 0$.

\vskip 0.5truecm
\noindent
{\it AMS} 2010 {\it subject classifications.} 35K05, 35K20.

\medskip
{\it Acknowledgements.} MvdB was supported by a Leverhulme Trust Emeritus Fellowship EM-2018-011-9. MvdB acknowledges hospitality by the Max Planck Institute for Mathematics, Bonn, and the Mathematical Institute, University of Neuch\^atel.
KG acknowledges support from the Max Planck Institute for Mathematics, Bonn, from October 2017 to July 2018.
The authors wish to thank the referee for helpful suggestions.

\end{abstract}

%%%%%%%%%%%%%%%%%%%%%%%%%%%%%%
\section{Introduction}\label{Introduction}
Let $D$ be an open, bounded set in $\R^m$ with finite Lebesgue measure $|D|$, and with boundary $\partial D$.
We consider the heat equation
\begin{equation*}%\label{e1}
\Delta u = \frac{\partial u}{\partial t},
\end{equation*}
and impose a Dirichlet $0$ boundary condition on $\partial D$. That is
\begin{equation*}%\label{e2}
u(x;t)=0,\,x\in \partial D,\, t>0.
\end{equation*}
We denote the (weak) solution corresponding to the initial datum
\begin{equation*}%\label{e3}
\lim_{t\downarrow 0}u(x;t)=1,\, x\in D,\,
\end{equation*}
by $u_D$. Then $u_D(x;t)$ represents the temperature at $x\in D$ at time $t$ when $D$ has initial temperature $1$, and its boundary is kept at fixed temperature $0$. The heat content of $D$ at $t$ is denoted by
\begin{equation*}%\label{e4}
Q_D(t)=\int_Ddx\,u_D(x;t).
\end{equation*}
Both $u_D$ and $Q_D(t)$ have been the subjects of a thorough investigation going back to the treatise by Carslaw and Jaeger, \cite{CarslawJ}. For more recent accounts we refer to \cite{MvdBEBD,PG}. 

Many different versions and extensions have already been considered.
For example, the case where $\partial D$ is smooth,
and $A$ is an open subset of $\partial D$ on which a Neumann (insulating) boundary condition
has been imposed, while the temperature $0$ Dirichlet condition has
been maintained on $\partial D-A$. This Zaremba boundary condition for the heat equation has been considered in
\cite{MvdBGKK}, for example. Even in the case where no boundary condition has been imposed on $\partial D$, the corresponding heat content, denoted by $H_D(t)$, has (if $\partial D$ is smooth) an asymptotic series as $t\downarrow 0$ similar to the one for $Q_D(t)$, see \cite{MvdBPG}, for example.

In this paper we consider the heat flow out of $D$ into $\R^m$, where a Dirichlet $0$ boundary condition has been imposed on a closed subset $\partial D_{-}\subset \partial D$, and where no boundary condition has been imposed on $\partial D_{+}:=\partial D-\partial D_{-}$. That is
\begin{equation}\label{e5}
\Delta u= \frac{\partial u}{\partial t},
\end{equation}
with boundary condition
\begin{equation}\label{e5a}
u(x;t)=0,\,x\in \partial D_{-},\, t>0.
\end{equation}
We denote the solution corresponding to the initial datum
\begin{equation}\label{e5b}
\lim_{t\downarrow 0}u(x;t)={\bf 1}_{D}(x),\,\textup{almost everywhere},
\end{equation}
by $u_{D,\partial D_{-}}$. Here ${\bf 1}_{D}$ is the indicator function of $D$.
Then $u_{D,\partial D_{-}}$ is the weak solution of \eqref{e5}, \eqref{e5a} and \eqref{e5b}, where \eqref{e5a} holds at all regular points of $\partial D_{-}$.
The open set $D$ looses heat via
two mechanisms: (i) part of the boundary, $\partial D_{-}$, is at fixed temperature $0$, and cools the interior of $D$; (ii) since the complement of $D$ is at initial temperature $0$, heat flows over the open part of the boundary, $\partial D_{+}$.
The corresponding heat content is denoted by
\begin{equation*}%\label{e6}
G_{D,\partial D_{-}}(t)=\int_D dx\, u_{D,\partial D_{-}}(x;t).
\end{equation*}
Let $A$ be a closed subset of $\R^m$, and let $p_{\R^m-A}(x,y;t), \, x\in \R^m-A,\, y\in \R^m-A,\, t>0$ be the heat kernel for the open set $\R^m-A$ with a Dirichlet $0$ boundary condition on $A$. This heat kernel is non-negative, symmetric in its space variables, and satisfies the heat semigroup property. Moreover, If $A$ and $B$ are closed subsets with $B\subset A$ then $p_{\R^m-A}(x,y;t)\le p_{\R^m-B}(x,y;t),\,x\in \R^m-A, y\in \R^m-A,\,t>0$. We refer to \cite{AG} for further details.
Then for $x\in D$
\begin{equation}\label{e7}
u_{D,\partial D_{-}}(x;t)=\int_D\, dy\,p_{\R^m-\partial D_{-}}(x,y;t).
\end{equation}
Let $\big(B(s), s\ge 0,\mathbb{P}_x,\, x\in \R^m\big)$ be Brownian motion associated with $\Delta$. Recall that $p_{\R^m-\partial D_{-}}$ is the transition density for Brownian motion on $\R^m$ with killing on $\partial D_{-}$. If $\tau_{\partial D_{-}}=\{\inf s\ge 0:B(s)\in \partial D_{-}\}$, then
\begin{equation*}%\label{e7a}
u_{D,\partial D_{-}}(x;t)=\mathbb{P}_x\big(\tau_{\partial D_{-}}\ge t, B(t)\in D\big),
\end{equation*}
which jibes with \eqref{e7}.

Since $D\subset \R^m-\partial D_{-}\subset\R^m$ , we have by monotonicity,
\begin{equation*}%\label{e8}
0\le u_D(x;t)=u_{D,\partial D}(x;t)\le u_{D,\partial D_{-}}(x;t)\le u_{D,\emptyset}(x;t).
\end{equation*}
Hence
\begin{equation}\label{e9}
Q_D(t)\le G_{D,\partial D_{-}}(t)\le H_D(t),\, t>0.
\end{equation}
Using the spectral resolution for the Dirichlet heat kernel on $\R^m-\partial D_{-}$ it is possible to show
that all three heat contents in \eqref{e9}
are strictly decreasing in $t$.
Moreover, \eqref{e5b}
implies that for $1\le p<\infty$,
\begin{equation}\label{e10}
\lim_{t\downarrow 0}\Vert u_{D,\partial D_{-}}(\cdot;t)-
{\bf 1}_{D}(\cdot)\Vert_{L^p(\R^m-\partial D_{-})}=0.
\end{equation}
The short proof below is instructive. See also \cite{MvdBKG1}. By monotonicity,
\begin{equation*}%\label{e11}
p_{\R^m-\partial D_{-}}(x,y;t)\le p_{\R^m}(x,y;t)= (4\pi t)^{-m/2}e^{-|x-y|^2/(4t)}.
\end{equation*}
Hence $0<u_{D,\partial D_{-}}(x;t)\le1$, and $|u_{D,\partial D_{-}}(x;t) - 1|\le 1$. Moreover,
\begin{align}\label{e12}
\lVert u_{D,\partial D_{-}}(\cdot;t)-
{\bf 1}_{D}(\cdot)\rVert_{L^p(\R^m)}^p&=\int_{D}dx \,
|u_{D,\partial D_{-}}(x;t)-1|^p+ \int_{\R^{m} - D}dx \, u_{D,\partial D_{-}}(x;t)^p
\nonumber \\ & \le \int_Ddx \, |u_{D,\partial D_{-}}(x;t)-1|+\int_{\R^m-D}dx \, u_{D,\partial D_{-}}(x;t) \nonumber \\  &
=\int_Ddx \, |u_{D,\partial D_{-}}(x;t)-1|+\int_{\R^m}dx \, u_{D,\partial D_{-}}(x;t)-\int_Ddx \,
u_{D,\partial D_{-}}(x;t).
\end{align}
By \eqref{e7}, Tonelli's theorem, and monotonicity,
\begin{equation}\label{e12a}
\int_{\R^m}dx \, u_{D,\partial D_{-}}(x;t)=\int_D\, dy\,\int_{\R^m}\,dx\,p_{\R^m-\partial D_{-}}(x,y;t)\le \int_D\, dy\,\int_{\R^m}\,dx\,p_{\R^m}(x,y;t)=\int_{D}dy.
\end{equation}
By \eqref{e12} and \eqref{e12a},
\begin{equation*}%\label{e12b}
\lVert u_{D,\partial D_{-}}(\cdot;t)-
{\bf 1}_{D}(\cdot)\rVert_{L^p(\R^m)}^p\le 2\int_Ddx \, |1-u_{D,\partial D_{-}}(x;t)|,
\end{equation*}
and \eqref{e10} follows by Lebesgue's Dominated Convergence theorem and \eqref{e5b}.

The main results of this paper are concerned with
the special case where $D$ is an open, bounded set in $\R^2$ with a polygonal boundary. Throughout we make the hypothesis that the vertices of $\partial D$ are the endpoints of exactly two edges, and that the collection of vertices $\mathcal{V}=\{V_1,V_2,\cdots\}$ is finite. We consider edges of two types: Dirichlet edges which include their endpoints, and open edges which include those vertices common to two open edges. The union of all Dirichlet edges, denoted by $\partial D_{-}$ as above, is a closed subset of $\R^2$, and we denote its length by $L(\partial D_{-})$. The union of all open edges, denoted by $\partial D_{+}$, is a relatively open subset of $\partial D$. We denote its length by $L(\partial D_{+})$. The length of $\partial D$ is given by
\begin{equation*}%\label{e13}
L(\partial D)=L(\partial D_{-})+L(\partial D_{+}).
\end{equation*}

It was shown in \cite{MvdBSS2} that if all edges are of Dirichlet type, then
\begin{equation}\label{e14}
Q_D(t)= |D|- \frac{2}{\pi^{1/2}}L(\partial D)t^{1/2} +
\sum_{\gamma\in \mathcal{C}} c(\gamma)t+ O(e^{-q_D/t}),\, t\downarrow 0,
\end{equation}
where $q_D > 0$ is a constant which
depends on $D$ only, $c: (0,2\pi] \to \R$ is defined by
\begin{equation}\label{e15}
c(\gamma)=\int_0^{\infty}d\theta\frac{4\sinh((\pi-\gamma)\theta)}{\sinh(\pi \theta)\cosh(\gamma \theta)},
\end{equation}
$\mathcal{C}=\{\gamma_1,\gamma_2,...\}$ are the interior angles at the vertices $V_1,V_2,...$,
and $L(\partial D)$ is the total length of all Dirichlet edges.

On the other hand, if all edges are of open type, that is $\partial D_{-}=\emptyset$, then it was shown in \cite{MvdBKG2} that
\begin{align}\label{e16}
H_{D}(t) &= \vert D \vert - \frac{1}{\pi^{1/2}}L(\partial D)t^{1/2} +
\sum_{\beta\in\mathcal{B}} b(\beta)t+O(e^{-h_D/t}),\, t\downarrow 0,
\end{align}
where $h_D > 0$ is a constant which
depends on $D$ only, $b: (0,2\pi) \rightarrow \R$ is defined by
\begin{equation*}%\label{e17}
b(\beta) =\begin{cases} \frac{1}{\pi} + \left(1-\frac{\beta}{\pi}\right)
\cot \beta, &\beta \in (0,\pi)\cup(\pi,2\pi);\\
0, &\beta=\pi, \end{cases},
\end{equation*}
$\mathcal{B}=\{\beta_1,\beta_2,...\}$ are the interior angles at the vertices $V_1,V_2,...$, and $L(\partial D)$ is the total length of all open edges.

The main result of this paper, Theorem \ref{the2} below, allows both open and Dirichlet edges. The collection of interior angles between two adjacent Dirichlet, respectively open, edges is denoted by $\mathcal{C}$, respectively $\mathcal{B}$. The collection of angles between an adjacent pair of open-Dirichlet edges (or Dirichlet-open edges) is denoted by $\mathcal{A}$ (see Figure~\ref{fig1}).
\begin{theorem}\label{the2} There exists a constant $g_D > 0$ depending on $D$ only such that
\begin{equation}\label{e18}
G_{D,\partial D_{-}}(t)=|D|-\frac{1}{\pi^{1/2}}\bigg(2L(\partial D_{-})+L(\partial D_{+})\bigg)t^{1/2}+\big(\sum_{\gamma\in \mathcal{C}}c(\gamma)+\sum_{\beta\in \mathcal{B}}b(\beta)+\sum_{\alpha\in \mathcal{A}}a(\alpha)\big)t+O(e^{-g_D/t}),\, t\downarrow 0,
\end{equation}
where $a :(0,2\pi)\mapsto \R$ is given by
\begin{equation}\label{e19}
a(\alpha)=-\frac34+\frac14\int_0^{\infty}d\theta\, \frac {4(\sinh((\pi-\frac{\alpha}{2})\theta))^2 - (\sinh((\pi-\alpha)\theta))^2}{\big(\sinh(\pi\theta/2)\big)^2 \cosh(\pi\theta)}.
\end{equation}
\end{theorem}

\begin{figure}[!ht]
 \begin{center}
 \begin{tikzpicture}
% Edges
\draw[dashed, ultra thick] (0,0) -- (0.5,1.75) -- (-2,2.5);
\draw[ultra thick] (-2,2.5) -- (2,5.5) -- (7,3);
\draw[dashed, ultra thick] (7,3) -- (5,0);
\draw[ultra thick] (5,0) -- (0,0);
% Sectors
\draw (0.75,0) arc(0:74.0546:0.75);
\draw (4.25,0) arc(180:56.31:0.75);
\draw (6.58397,2.37596) arc(236.3099:153.4349:0.75);
\draw (1.4,5.05) arc(-143.1301:-26.5651:0.75);
\draw (0.29396,1.02886) arc(-105.9454:163.30076:0.75);
\draw (-1.2816,2.2845) arc(-16.69924:36.86897:0.75);
% Labels
\node at (3.5,2.5) {$D$};
\node at (0.9,0.5) {$\alpha_1$};
\node at (1.5,1.95) {$\beta_1$};
\node at (-0.9,2.65) {$\alpha_2$};
\node at (2,4.5) {$\gamma_1$};
\node at (6,2.75) {$\alpha_3$};
\node at (4.5,0.85) {$\alpha_4$};
\end{tikzpicture}
 \caption{An open set $D \subset \R^2$ with polygonal boundary:
 the Dirichlet, respectively open, edges are displayed as solid, respectively dashed, lines.}
 \label{fig1}
  \end{center}
 \end{figure}
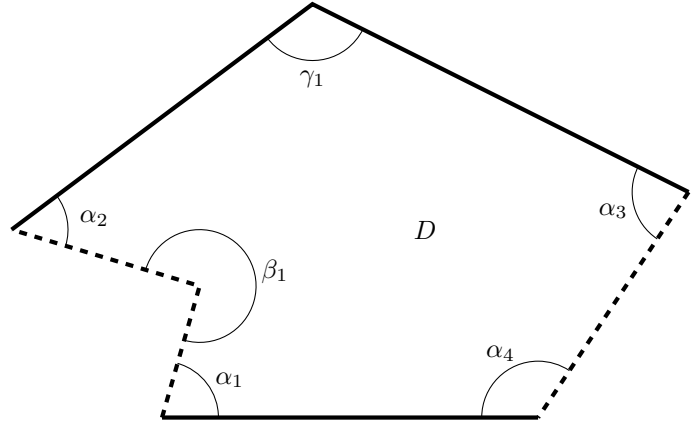

The main results of both \cite{MvdBKG2} and \cite{MvdBSS2} hold for more general polygons. For example, vertices with just one edge or more than two are allowed. If a vertex supports just one edge, then the corresponding angle equals $2\pi$
and will contribute $c(2\pi)$ to the coefficient of $t$ in \eqref{e14}. That edge counts double in the total length of Dirichlet edges. Indeed, that edge cools $D$ at both sides. In general, the contribution from
the angles to the coefficient $t$ in \eqref{e14} is additive. The Dirichlet condition on
the edges implies this additivity. That does not hold true in the setting of open edges. If two wedges with angles, say $\beta_1$ and $\beta_2$, are supported by the same vertex, then there is an additional contribution to the coefficient of $t$, depending on $\beta_1, \beta_2$ and the angle between these two wedges (see \cite{MvdBKG2}). Furthermore, if a vertex supports just one edge, then the corresponding angle, and the corresponding edge contribute $0$ to the coefficients of $t$ and $t^{1/2}$ respectively. Indeed, heat does not flow over this edge into $\R^2-D$. We shall not consider these cases, and we assume that each vertex supports precisely two edges.

The proof of Theorem \ref{the2} is based on a partition of $D$ combined with model computations, as are the proofs of \eqref{e14} and \eqref{e16}. The main computation is the one for circular sectors with radius $R$ with
opening angles $\gamma, \beta, \alpha$ depending on whether one deals with a Dirichlet-Dirichlet wedge, an open-open wedge, or, as in this paper, a Dirichlet-open wedge. The geometry of the Dirichlet-open wedge is one edge on which a Dirichlet boundary condition has been imposed, and an open edge separated by
angle $\alpha$ (see Figure~\ref{fig2}). Our main result for such a circular sector is the following.

\begin{figure}[!ht]
 \begin{center}
 \begin{tikzpicture}
% Edges
\draw[dashed, ultra thick] (0,0) -- (4,3);
\draw[ultra thick] (0,0) -- (5,0);
% Sectors
\draw (1,0) arc(0:36.86989765:1);
\node at (0.5,-0.3) {$R$};
\node at (1.2,0.35) {$\alpha$};
\end{tikzpicture}
 \caption{A Dirichlet-open wedge with angle $\alpha$: the Dirichlet, respectively open, edge is displayed as a solid, respectively dashed, line.}
 \label{fig2}
  \end{center}
 \end{figure}
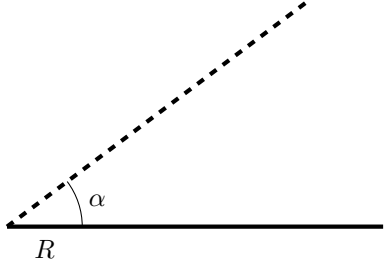

\begin{theorem}\label{the3}
Let $W_{\alpha}=\{(r,\phi): r>0,\, 0<\phi<\alpha\}$ in polar coordinates, and let $ u_{W_{\alpha}}((r,\phi);t)$ be the solution of the heat equation with a Dirichlet $0$ boundary condition on the positive $x_1$ axis, and initial data ${\bf 1}_{W_{\alpha}}$. Then, in polar coordinates, we have
\begin{align}\label{e20}
\int_0^{R}  dr\,r\,\int_0^{\alpha}d\phi \,  u_{W_{\alpha}}((r,\phi);t)=&\frac12\alpha R^2-\frac{3}{\pi^{1/2}}Rt^{1/2}+a(\alpha)t\nonumber \\ &+\frac{3Rt^{1/2}}{\pi^{1/2}}\int_1^{\infty}\frac{dv}{v^2}\,\int_0^1 d\zeta\frac{\zeta}{(1-\zeta^2)^{1/2}}e^{-R^2\zeta^2v^2/(4t)}\nonumber\\ &+O(e^{-m_{\alpha}R^2  / (4t)}),\, t \downarrow 0,
\end{align}
where $m_{\alpha} >0$ is a constant which depends on $\alpha$ only.
\end{theorem}
We recognise the various terms in the right-hand side as follows. The first term is the area of the circular sector with opening angle $\alpha$ and radius $R$. The second term combines the contributions from an open edge of length $R$, and a Dirichlet edge of length $R$. The latter having an extra factor $2$. The third term is the angle contribution. The fourth term represents the contribution from two cusps. See Section~\ref{sec3} for details.

Unlike the integral for $c(\gamma)$ in \eqref{e15}, it is possible to evaluate the expression for $a(\alpha)$ in \eqref{e19}. To do so we write
\begin{equation*}%\label{e21}
\frac{1}{\big(\sinh(\pi\theta/2\big))^2\cosh(\pi\theta)}=\frac{1}{(\sinh(\pi\theta/2))^2}-\frac{2}{\cosh(\pi\theta)},
\end{equation*}
and compute the resulting four integrals using formulae 3.511.7 and 3.511.9 in \cite{GR}. The common range of convergence for these four integrals is
$\pi<\alpha<3\pi/2$. We find
\begin{equation*}%\label{e22}
a(\alpha)=-\frac38+\frac{3}{4\pi}-\frac{1}{8\cos \alpha}+\frac{1}{2\cos(\alpha/2)}+\bigg(\frac{7}{4}-\frac{3\alpha}{4\pi}\bigg)\frac{1}{\tan \alpha}+\bigg(\frac{1}{4}-\frac{\alpha}{4\pi}\bigg)\tan\alpha,\, \pi<\alpha<3\pi/2.
\end{equation*}
Outside this interval we can use \eqref{e19} to evaluate $a(\alpha)$. For example, we have
\begin{equation*}%\label{e23}
a(\pi/2)=-\frac38+\frac1\pi+\frac12\sqrt 2,
\end{equation*}
\begin{equation}\label{e24}
a(\pi)=-\frac14,
\end{equation}
\begin{equation*}%\label{e25}
a(3\pi/2)=-\frac38+\frac1\pi-\frac12\sqrt 2.
\end{equation*}

The value $a(\pi)=-\frac14$ in \eqref{e24} is of particular interest. Consider an open, bounded set $D$ in $\R^m$ with $C^{\infty}$ boundary $\partial D$.
Let $\partial D_{-}$ be a closed subset of $\partial D$ with $C^{\infty}$ boundary $\Sigma$, and $\dim \Sigma=m-2$.
Let $u_{D,\partial D_{-}}$ be the solution of \eqref{e5}, \eqref{e5a}, and \eqref{e5b}.
Then, provided an asymptotic series in half powers of $t$ exists, we have
\begin{align}\label{e30}
G_{D,\partial D_{-}}(t)=&|D|-\pi^{-1/2}\bigg(2\int_{\partial D_{-}}d\sigma+\int_{\partial D -\partial D_{-}}d\sigma\bigg)t^{1/2}\nonumber \\ &+\bigg(\frac12\int_{\partial D_{-}}d\sigma L_{aa}(\sigma)-\frac14 \textup{vol}(\Sigma)\bigg)t+O(t^{3/2}),\, t\downarrow 0,
\end{align}
where $d\sigma$ denotes the surface measure on $\partial D$, $L_{aa}$ is the trace of the second fundamental form
defined by the inward unit normal vector field of $\partial D$ in $D$,
$\textup{vol}(\Sigma)$ is the $(m-2)$-dimensional volume of the boundary of $\partial D_{-}$ in $\partial D$, and
$a(\pi)$ is its coefficient. To see that \eqref{e30} holds, we note that the local geometry around $\Sigma$ is as follows.
Let $P$ be a point of $\Sigma$. Then straightening out the boundary of $\partial D$ around $P$ we obtain, locally, an
$(m-1)$-dimensional hyper plane. The straightening out of $\Sigma$ around $P$ partitions this hyper plane into two
hyper half-planes at angle $\pi$. On one (closed) hyper half-plane we have a Dirichlet $0$ boundary condition,
and on the remaining open hyper half-plane we do not have boundary conditions.
This is precisely the geometry of a Dirichlet-open wedge with angle $\pi$ times $\Sigma$.
This then leads to the $a(\pi)\textup{vol}(\Sigma) t$ contribution in \eqref{e30}. The computation of the
coefficient of $t^{3/2}$ promises to be more complicated even in this special setting.
One expects that there is an integral over $\Sigma$ involving both the second fundamental form of $\Sigma$ in
$D$ and the second fundamental form of $\partial D$ in $D$. Consequently,
several special case calculations would be required. See also \cite{MvdBGKK}.

The proofs of Theorems \ref{the3}
and \ref{the2} have been deferred to Sections \ref{sec2}, and \ref{sec3} respectively. In Section \ref{sec4} we state some technical preliminaries which will be used in the proof of Theorem \ref{the3}.

\section{Proof of Theorem \ref{the2}}\label{sec3}
In this section, we make use of Theorem \ref{the3}. We prove that the latter theorem holds in Section \ref{sec2}.

Kac's principle of not feeling the boundary asserts that the solution of the heat equation with initial datum ${\bf 1}_D$, where $D$ is an open set in $\R^m$, is equal to $1$ on the interior of $D$ up to an exponentially small remainder, as $t\downarrow 0$. Kac formulated his principle in the case where a Dirichlet $0$ boundary condition is imposed on all of $\partial D$, that is $\partial D_{+}=\emptyset$. It has been shown that it also holds if no boundary condition is
imposed on $\partial D$, that is $\partial D_{-}=\emptyset$. See, for example, Proposition 9(i) in \cite{MvdB1}. In the same spirit, we have the following lemma.
\begin{lemma}\label{lem1}
If $D$ is an open set in $\R^2$, and if $\partial D_{-}$ is a closed subset of $\partial D$, then
\begin{equation}\label{e68}
1\ge \int_D dy\,  p_{\R^2-\partial D_{-}}(x,y;t)\ge \int_D dy\,  p_{D}(x,y;t)\ge 1- 2e^{-d(x,\partial D)^2/(4t)}.
\end{equation}
\end{lemma}
\begin{proof}
Since the Dirichlet heat kernel is monotone in the domain, and since $D\subset \R^m-\partial D_{-}$,
\begin{equation*}%\label{e69}
p_{\R^m-\partial D_{-}}(x,y;t)\ge p_{D}(x,y;t)>0,\, x\in D,\, y\in D,\, t>0.
\end{equation*}
Hence $u_{D, \partial D_{-}}(x;t) \ge \int_Ddy\,p_D(x,y;t)$. The latter integral has been bounded from below in Lemma 4 of \cite{MvdBSS2}. Taking $m=2$ in the first line of (3.2) in that paper we find \eqref{e68}.
The upper bound in \eqref{e68} follows as $p_{\R^m-\partial D_{-}}(x,y;t) \le p_{\R^m}(x,y;t)$
and $\int_{\R^m} dy \, p_{\R^m}(x,y;t) =1$.
\end{proof}

As in \cite{MvdBSS1, MvdBSS2, MvdBKG2}, the strategy of the proof of Theorem \ref{the2} is to partition $D$ into sets on which $u_{D,\partial D_{-}}(x;t)$ is approximated either by $1$, or by $u_{W_{\alpha}}(x;t)$, or by $u_{H}(x;t)$ (where $H \subset \R^m$ is a half-space) depending on where $x \in D$ lies with respect to the partition. By Lemma \ref{lem1}, the terms which compensate for these approximations are exponentially small.

Below we describe the partition of the set $D$. At each vertex of $\partial D$ with angle $\theta$, we consider the circular sector of radius $R>0$ and angle $\theta$ that is contained in $D$.
For $\delta >0$ (to be specified later), we consider the set of points in $D$ that are at distance less than $\delta$ from $\partial D$ and that are not contained in the union of the circular sectors (see Figure \ref{fig3}).

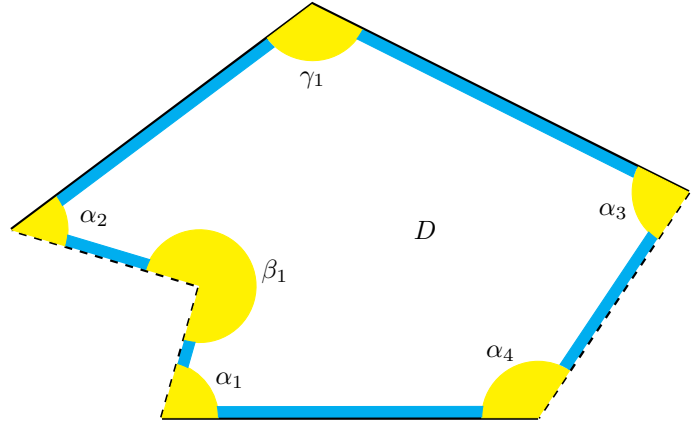
\begin{figure}[!ht]
\begin{center}
\begin{tikzpicture}
% Edges
\draw[dashed, ultra thick] (0,0) -- (0.5,1.75) -- (-2,2.5);
\draw[ultra thick] (-2,2.5) -- (2,5.5) -- (7,3);
\draw[dashed, ultra thick] (7,3) -- (5,0);
\draw[ultra thick] (5,0) -- (0,0);
% Sectors
\draw[fill=yellow,yellow] (0,0) -- (0.75,0) arc(0:74.0546:0.75) -- cycle;
\draw[fill=yellow,yellow] (5,0) --  (4.25,0) arc(180:56.31:0.75) -- cycle;
\draw[fill=yellow,yellow] (7,3) -- (6.58397,2.37596) arc(236.3099:153.4349:0.75) -- cycle;
\draw[fill=yellow,yellow] (2,5.5) -- (1.4,5.05) arc(-143.1301:-26.5651:0.75);
\draw[fill=yellow,yellow] (0.5,1.75) -- (0.29396,1.02886) arc(-105.9454:163.30076:0.75);
\draw[fill=yellow,yellow] (-2,2.5) -- (-1.2816,2.2845) arc(-16.69924:36.86897:0.75);
% Region close to boundary
\draw[fill=cyan,cyan] (0.75,0) arc(0:11.3099:0.75) -- (4.2652,0.15)
arc(168.6901:180:0.75) (4.25,0)--cycle;
\draw[fill=cyan, cyan] (6.584,2.376) arc(-123.6901:-135:0.75) --(5.2912,0.7072)
arc(67.6199:56.31:0.75) (5.416,0.624) -- cycle;
\draw[fill=cyan, cyan] (2.6708,5.1646) arc(-26.5651:-37.875:0.75) --(6.2621,3.2012)
arc(164.7448:153.4349:0.75) (6.3292,3.3354) --cycle;
\draw[fill=cyan, cyan] (-1.4,2.95) arc(36.8699:25.56:0.75) --(1.49,4.93)
arc(-131.8202:-143.1301:0.75) (1.4,5.05) -- cycle;
\draw[fill=cyan, cyan] (-1.2816,2.2845) arc(-16.69924:-5.38931:0.75) --(-0.16215,2.10221)
arc(151.9908:163.30076:0.75) (-0.21834,1.96551) -- cycle;
\draw[fill=cyan, cyan] (0.29396,1.02886) arc(-105.945395:-94.635463:0.75)  --(0.34347,0.66673)
 arc(62.74467:74.0546:0.75) (0.20604,0.72114) -- cycle;
% Labels
\node at (3.5,2.5) {$D$};
\node at (0.9,0.5) {$\alpha_1$};
\node at (1.5,1.95) {$\beta_1$};
\node at (-0.9,2.65) {$\alpha_2$};
\node at (2,4.5) {$\gamma_1$};
\node at (6,2.75) {$\alpha_3$};
\node at (4.5,0.85) {$\alpha_4$};
\end{tikzpicture}
\caption{Partition of $D$ (the Dirichlet, respectively open, edges are displayed as solid, respectively dashed, lines).}
\label{fig3}
\end{center}
\end{figure}

Let $H=\{(x_1,x_2)\in \R^2: x_2>0\}$. Up to changing the coordinates (if necessary), we can suppose that $\partial D \cap \partial H$ is an edge $e$ of length $\ell$. Let $L = \ell - 2R$. In this way, each blue region in Figure \ref{fig3} can be written (up to a set of measure $0$) as the union of a rectangle
$$ \{(x_1,x_2) \in \R^2 : R < x_1 < R+L, 0<x_2 <\delta\},$$
and two cusps of the form
\begin{equation*}
E_1(\delta,R)=\{x\in \R^2: 0<x_1 <R,\,|x|>R,\, 0<x_2<\delta\},
\end{equation*}
and
\begin{equation*}
E_2(\delta,R)=\{x\in \R^2: \ell-R<x_1 <\ell,\, |x - (\ell,0)|>R,\,0<x_2<\delta\}.
\end{equation*}
We say that these cusps are adjacent to $\partial H$.

We observe that each sector has two neighbouring cusps. In the partition of $D$, cusps of two types feature. That is, those cusps adjacent to $\partial H$ with a Dirichlet $0$ boundary condition on $e$, and those cusps adjacent to $\partial H$ without a boundary condition on $e$ (see Figure \ref{fig4}). Cusps of the latter type feature in \cite{MvdBKG2}, and those of the former type feature in \cite{MvdBSS2}.

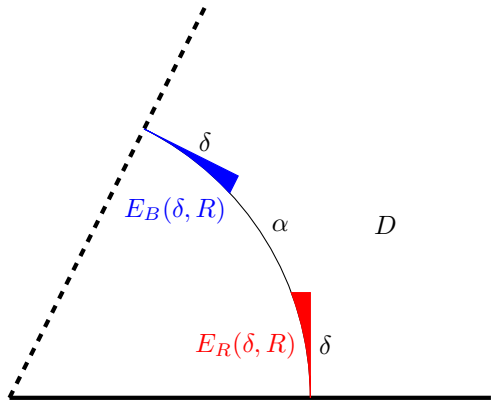
\begin{figure}[!ht]
 \begin{center}
 \begin{tikzpicture}[scale=2]
% Edges
\draw[dashed, ultra thick] (0,0) -- (1.3,2.6);
\draw[ultra thick] (0,0) -- (3.2,0);
% Sector
\draw (2,0) arc(0:63.43495:2);
% Cusps
\draw[fill=red, red] (2,0) arc(0:20.48732:2) -- (2,0.7)--cycle;
\draw[fill=blue, blue] (0.89443,1.78885) arc(63.43495:42.94762882:2) -- (1.52053,1.475805)--cycle;
%% Labels
\node at (2.1,0.35) {$\delta$};
\node at (1.3,1.7) {$\delta$};
\node at (1.8,1.15) {$\alpha$};
\node at (2.5,1.15) {$D$};
\node[red] at (1.57,0.35) {$E_{R}(\delta,R)$};
\node[blue] at (1.1,1.25) {$E_{B}(\delta,R)$};
\end{tikzpicture}
 \caption{ A sector contained in a Dirichlet-open wedge with angle $\alpha$, and its neighbouring cusps $E_{R}(\delta,R)$ adjacent to an edge with a Dirichlet 0 boundary condition, and $E_{B}(\delta,R)$ adjacent to an open edge (the Dirichlet, respectively open, edge is displayed as a solid, respectively dashed, line).}
 \label{fig4}
  \end{center}
 \end{figure}

We first consider the case of a cusp which is adjacent to $\partial H$ with a Dirichlet 0 boundary condition.
\begin{lemma}\label{lem2}
If $\delta<R$ then
\begin{equation*}%\label{e70}
\int_{E(\delta,R)}dx\, u_H(x;t)=|E(\delta,R)|-\frac{2Rt^{1/2}}{\pi^{1/2}}\int_1^{\infty}\frac{dw}{w^2}\int_0^1\frac{vdv}{(1-v^2)^{1/2}}e^{-R^2v^2w^2/(4t)}
+O(t^{1/2}e^{-\delta^2/(4t)}).
\end{equation*}
\end{lemma}
\begin{proof} See also (4.7) in \cite{MvdBSS2}.
We have that
\begin{equation*}%\label{e71}
u_H(x;t)=\frac{1}{(\pi t)^{1/2}}\int_0^{x_2}dq\,e^{-q^2/(4t)}.
\end{equation*}
Since the length of the line segment in $E(\delta,R)$ parallel to the $x_1$ axis equals $R-(R^2-x_2^2)^{1/2}$, we have
\begin{align}\label{e72}
\int_{E(\delta,R)}dx\,u_H(x;t)&=\frac{1}{(\pi t)^{1/2}}\int_0^{\delta}dx_2 ( R-(R^2-x_2^2)^{1/2}) \int_0^{x_2}dq\,e^{-q^2/(4t)}\nonumber \\&
=|E(\delta,R)|-\frac{1}{(\pi t)^{1/2}}\int_0^{\delta}dx_2 \big(R-(R^2-x_2^2)^{1/2}\big)\int_{x_2}^{\infty}dq\,e^{-q^2/(4t)}\nonumber \\ &
=|E(\delta,R)|-\frac{1}{(\pi t)^{1/2}}\int_0^{\delta}dx_2 \big(R-(R^2-x_2^2)^{1/2}\big)x_2\int_{1}^{\infty}dw\,e^{-w^2x_2^2/(4t)}\nonumber \\ &
=|E(\delta,R)|-\frac{2t^{1/2}}{\pi ^{1/2}}\int_{1}^{\infty}\frac{dw}{w^2}\int_0^{\delta}\frac{x_2dx_2}{\big(R^2-x_2^2\big)^{1/2}}e^{-w^2x_2^2/(4t)}\nonumber \\ &\hspace{18mm}+\frac{2t^{1/2}}{\pi ^{1/2}}\big(R-(R^2-\delta^2)^{1/2}\big)\int_{1}^{\infty}\frac{dw}{w^2}e^{-w^2\delta^2/(4t)}\nonumber \\ &
=|E(\delta,R)|-\frac{2Rt^{1/2}}{\pi^{1/2}}\int_1^{\infty}\frac{dw}{w^2}\int_0^1\frac{vdv}{(1-v^2)^{1/2}}e^{-R^2v^2w^2/(4t)}\nonumber \\ &
\hspace{18mm}+\frac{2t^{1/2}}{\pi ^{1/2}}\big(R-(R^2-\delta^2)^{1/2}\big)\int_{1}^{\infty}\frac{dw}{w^2}e^{-w^2\delta^2/(4t)}\nonumber \\ &\hspace{18mm}
+\frac{2t^{1/2}}{\pi ^{1/2}}\int_{1}^{\infty}\frac{dw}{w^2}\int_{\delta}^R\frac{x_2dx_2}{\big(R^2-x_2^2\big)^{1/2}}e^{-w^2x_2^2/(4t)}.
\end{align}
Both the third and fourth terms in the right-hand side of \eqref{e72} are $O(t^{1/2}e^{-\delta^2/(4t)})$.
\end{proof}

Next we consider the case of a cusp which is adjacent to $\partial H$ which is open (that is without a boundary condition).
\begin{lemma}\label{lem3}
If $\delta<R$, then
\begin{equation}\label{e73}
\int_{E(\delta,R)}dx\, u_{H,\emptyset}(x;t)=|E(\delta,R)|-\frac{Rt^{1/2}}{\pi^{1/2}}\int_1^{\infty}\frac{dw}{w^2}\int_0^1\frac{vdv}{(1-v^2)^{1/2}}
e^{-R^2v^2w^2/(4t)}+O(t^{1/2}e^{-\delta^2/(4t)}).
\end{equation}
\end{lemma}
\begin{proof}
We recall that for $D \subset \R^m$ open,
\begin{equation*}
  u_{D,\emptyset}(x;t) = \int_{D} dy \, (4\pi t)^{-m/2}e^{-|x-y|^2/(4t)},
\end{equation*}
(see \cite{MvdB1} for example). Hence, for $D=H$, we have
\begin{equation}\label{e74}
u_{H,\emptyset}(x;t)=1-\frac{1}{(4\pi t)^{1/2}}\int_{x_2}^{\infty}dq\,e^{-q^2/(4t)}.
\end{equation}
Comparing \eqref{e74} with
\begin{equation*}%\label{e75}
u_{H}(x;t)=1-\frac{1}{(\pi t)^{1/2}}\int_{x_2}^{\infty}dq\,e^{-q^2/(4t)},
\end{equation*}
we see that the second, third and fourth terms in the right-hand side of \eqref{e72} are weighted with a factor $\frac12$ in the computation of the integral in the left-hand side of \eqref{e73}. This then gives \eqref{e73}.
\end{proof}
\begin{lemma}\label{lem4}
If $S=\{(x_1,x_2)\in\R^2:0<x_1<L,\, 0<x_2<\delta\}$, then
\begin{equation}\label{e76}
\int_S dx\,u_H(x;t)=|S|-\frac{2Lt^{1/2}}{\pi^{1/2}}+O(t^{1/2}e^{-\delta^2/(4t)}),
\end{equation}
and
\begin{equation}\label{e77}
\int_S dx\,u_{H,\emptyset}(x;t)=|S|-\frac{Lt^{1/2}}{\pi^{1/2}}+O(t^{1/2}e^{-\delta^2/(4t)}).
\end{equation}
\end{lemma}
\begin{proof}
We have
\begin{align*}%\label{e78}
\int_S dx\,u_H(x;t)&=\int_0^L dx_1\,\int_0^{\delta}dx_2\,\big(1-\frac{1}{(\pi t)^{1/2}}\int_{x_2}^{\infty}dq\,e^{-q^2/(4t)}\big)\nonumber \\ &
=|S|-\frac{2Lt^{1/2}}{\pi^{1/2}}+\int_{\delta}^{\infty}dx_2\,\frac{L}{(\pi t)^{1/2}}\int_{x_2}^{\infty}dq\,e^{-q^2/(4t)}\nonumber \\ &
=|S|-\frac{2Lt^{1/2}}{\pi^{1/2}}+O(t^{1/2}e^{-\delta^2/(4t)}).
\end{align*}
This proves \eqref{e76}. The observation concluding the proof of Lemma \ref{lem3} immediately implies \eqref{e77}.
\end{proof}
{\it Proof of Theorem \ref{the2}.}
Similarly to the strategies of the proofs in \cite{MvdBSS1,MvdBSS2,MvdBKG2}, it remains to apply the
model computations in Lemmas \ref{lem2}, \ref{lem3} and \ref{lem4}, the sector computations from Theorem \ref{the3}, \cite{MvdBKG2} and \cite{MvdBSS2} to the sets which partition $D$, and then apply Lemma \ref{lem1} to the compensating terms.

We first choose $R$ and $\delta$ appropriately in the partition of $D$. Let $v$ be an arbitrary vertex of the polygonal boundary, and let $e_v$ denote the union of the two edges of $\partial D$ adjacent to $v$. We choose
\begin{equation*}%\label{e79}
R=\frac12 \inf_{v\in \mathcal{V}}\inf\{d(v,y):y\in \partial D-e_v\}.
\end{equation*}
This choice of $R$ guarantees that all circular sectors are non-overlapping. Moreover, the distance from any point in a circular sector with vertex $v$, radius $R$, and angle $\theta$ to $W_{\theta} - D$ is at least $R$.
By Lemma \ref{lem1} we have that the model computations for the sectors with angles in $\mathcal{A},\mathcal{B},\mathcal{C}$ give the appropriate contributions to $G_{D,\partial D_{-}}(t)$ in \eqref{e18} up to an additive constant which is bounded in absolute value by $2|D|e^{-R^2/(4t)}$.

Next we choose $\delta$ sufficiently small to ensure that the cusps are pairwise disjoint.
We define $\varepsilon$ to be the smallest interior angle of the boundary $\partial D$:
\begin{equation*}%\label{e80}
\varepsilon=\min\{\mathcal{A}\cup\mathcal{B}\cup\mathcal{C}\}.
\end{equation*}
It is straightforward to check that
\begin{equation*}%\label{e81}
\delta=\frac{R}{2}\sin(\varepsilon/2)
\end{equation*}
satisfies the aforementioned condition.

The distance between the cusp and $H - D$ is larger than $\delta = \frac{R}{2}\sin(\varepsilon/2)$ (if we consider the cusp corresponding to the sector with angle $\varepsilon$).
By Lemma \ref{lem1}, we have that the model computations in Lemma \ref{lem2} and Lemma \ref{lem3} give the appropriate contributions to $G_{D,\partial D_{-}}(t)$ up to an additive constant which is bounded in absolute value by $2|D|e^{-R^2(\sin(\varepsilon/2))^2/(16t)}$. This is because the terms of order $t^{3/2}$ and higher in Theorem \ref{the3}, Lemma \ref{lem2} and Lemma \ref{lem3} cancel out up to an exponentially small remainder.

Next we consider the contribution of the subset of $D$ which is within distance $\delta$ of $\partial D$, and which is not contained in any of the radial sectors and their corresponding cusps. This subset is a collection of disjoint rectangles supported either by a Dirichlet or an open edge respectively. Each such rectangle has at least distance $\delta$ to any of the other edges. We conclude that, by Lemma \ref{lem4} and Lemma \ref{lem1}, they give the various contributions to $G_{D,\partial D_{-}}(t)$ up to an additive constant which is bounded in absolute value by $2|D|e^{-R^2(\sin(\varepsilon/2))^2/(16t)}$.

The remaining subset of $D$ which is not contained in a sector, cusp or rectangle has distance $\delta$ to the boundary, and so contributes its measure up to an additive constant which is bounded in absolute value by $2|D|e^{-R^2(\sin(\varepsilon/2))^2/(16t)}$, by Lemma \ref{lem1}.
All remainders above and in the proof of Theorem \ref{the3} are of the form $O(t^{\zeta} e^{-\eta/t}),\zeta\ge 0,\eta>0$. This gives the remainder in \eqref{e18}.
\hspace*{\fill }$\square $

\section{Technical preliminaries}\label{sec4}
It has been noted (see p.43 in \cite{MvdBSS2}) that there are three closed form expressions for the heat kernel of a wedge with opening angle $\gamma$, see \cite{CJ}, \cite{KM}, and \cite{SF}. The authors of \cite{MvdBSS2} were unable to extract the angle contribution $c(\gamma)t$ featuring in \eqref{e15} from these expressions. In the case at hand, there is a fourth explicit formula for the heat kernel of a wedge with opening angle $2\pi$ (see p.380 in \cite{CarslawJ}). However, we were unable to obtain a workable expression using that formula.

D. B. Ray managed to compute the angle contribution of the trace of the Dirichlet heat semigroup for a polygon using the Laplace transform of the heat kernel for a wedge, expressed as a Kontorovich Lebedev transform (see the footnote on p.44 of \cite{MKS}). This strategy has been successfully employed in both \cite{MvdBSS1} and \cite{MvdBSS2}. We also employ it in this article.

Let $W_{\alpha}$ be the open infinite wedge as in Theorem \ref{the3}, and let $p_{W_{\alpha}}(A_1,A_2;t)$ denote the Dirichlet heat kernel for $W_{\alpha}$.
Throughout we require $s>0,t>0$. Let
\begin{equation*}%\label{e31}
\hat{p}_{W_{\alpha}}(A_1,A_2;s)=\int_{0}^{\infty}dt \, e^{-st}p_{W_{\alpha}}(A_1,A_2;t)
\end{equation*}
be the associated Green's function (that is, the Laplace transform of $p_{W_{\alpha}}(A_1,A_2;t)$), and let $A_i=(a_i,\alpha_i),\, i=1,2$ in polar coordinates.
Then, following the footnote on p.44 in \cite{MKS}, and Appendix A of \cite{NRS},
\begin{align}\label{e32a}
\hat{p}&_{W_{ \alpha}}(A_1,A_2;s)=\frac{1}{\pi^2}\int_0^{\infty}d\theta K_{i\theta}(\sqrt s a_1)K_{i\theta}(\sqrt s a_2)\nonumber \\ &\times\left(\cosh((\pi-|\alpha_1-\alpha_2|)\theta)-\frac{\sinh(\pi\theta)}{\sinh(\alpha \theta)}\cosh((\alpha-\alpha_1-\alpha_2)\theta)+\frac{\sinh((\pi-\alpha)\theta)}{\sinh(\alpha\theta)}\cosh((\alpha_1-\alpha_2)\theta)\right),
\end{align}
where $K_{i\theta}$ is the modified Bessel function, defined for example by
formula 3.547.4 of \cite{GR},
\begin{equation}\label{e42}
K_{i\theta}(\sqrt s a)=\int_0^{\infty} dw \cos(w\theta) e^{-\sqrt s a\cosh w}.
\end{equation}
In the special case $\alpha=2\pi$, \eqref{e32a} simplifies and we obtain
 \begin{align*}%\label{e32}
\hat{p}_{W_{2\pi}}(A_1,A_2;s)&=\frac{1}{\pi^2}\int_0^{\infty}d\theta K_{i\theta}(\sqrt s a_1)K_{i\theta}(\sqrt s a_2)\nonumber \\ &\times\left(\cosh((\pi-|\alpha_1-\alpha_2|)\theta)-\frac{\sinh(\pi\theta)}{\sinh(2\pi \theta)}\left(\cosh((2\pi-\alpha_1-\alpha_2)\theta)+\cosh((\alpha_1-\alpha_2)\theta)\right)\right).
\end{align*}
In order to prove Theorem \ref{the3} in Section \ref{sec2} below, we compute
\begin{equation}\label{e32b}
\int_{0}^{R} da_1 \, a_1 \int_{0}^{\infty} da_2 \, a_2 \int_{0}^{\alpha} d\alpha_1 \int_{0}^{\alpha} d\alpha_2
\, \hat{p}_{W_{2\pi}}(A_1,A_2;s),
\end{equation}
and then take the inverse Laplace transform. Throughout this paper we denote by $L^{-1}$ the inverse Laplace transform. That is, if $\hat{f}(s)=\int_0^{\infty}dt\, e^{-st}f(t)$ then $L^{-1} \{\hat{f}\}
(t)=f(t)$, at points of continuity of $f$.

The lack of a suitable Tauberian theorem prevents us from deducing the behaviour as $t\downarrow 0$ of $\int_0^{R} dr\,r\,\int_0^{\alpha}d\phi \, u_{W_{\alpha}}((r,\phi);t)$ from the behaviour as $s\uparrow\infty$ of the expression under \eqref{e32b}. So after the computation of \eqref{e32b}, the resulting $s$-dependent terms have to be inverted to the $t$-domain, including those terms which turn out to be exponentially small in $t$. For the reader's convenience, we list some relevant formulae for the computation of \eqref{e32b} above.

Formulae 6.561.16, 8.332.3 in \cite{GR} yield
\begin{equation}\label{e34}
\int_0^{\infty}da\, a K_{i\theta}(\sqrt s a)=\frac{\pi\theta}{2s\sinh(\pi\theta/2)}.
\end{equation}
Moreover formulae 6.794.2, 6.795.1, 4.114.2, 4.116.2 in \cite{GR} read
\begin{equation}\label{e35}
\int_0^{\infty}d\theta\,\cosh(\pi\theta/2)  K_{i\theta}(\sqrt s a)=\frac{\pi}{2},\, a>0,
\end{equation}
\begin{equation}\label{e35a}
\int_0^{\infty}d\theta\, \cos(b\theta)K_{i\theta}(a)=\frac{\pi}{2}e^{-a\cosh b},\,a>0, |\Im b|\le \frac{\pi}{2},
\end{equation}
\begin{equation}\label{e35b}
\int_0^{\infty}d\theta\, \frac{\cos(a\theta)}{\theta}\frac{\sinh(\beta \theta)}{\cosh(\gamma \theta)}=\frac12 \log\bigg(\frac{\cosh(a\pi/(2\gamma))+\sin(\beta\pi/(2\gamma))}{\cosh(a\pi/(2\gamma))-\sin(\beta\pi/(2\gamma))}\bigg),|\Re \beta|<\Re \gamma.
\end{equation}
\begin{equation}\label{e35c}
\int_0^{\infty} \frac{d\theta}{\theta}\, \cos(a\theta)\tanh(\beta\theta)=\log\coth\big(a\pi/(4\beta)\big),
\end{equation}
where $\Re \beta$, respectively $\Im \beta$, denotes the real, respectively imaginary, part of $\beta$.

Finally, formula 5.6.3 in \cite{E} reads
\begin{equation}\label{e35d}
L^{-1}\{s^{-1}e^{-(as)^{1/2}}\}(t)=\frac{2}{\pi^{1/2}}\int_{(a/(4t))^{1/2}}^{ \infty}dr\,e^{-r^2}.
\end{equation}

\section{Proof of Theorem \ref{the3}}\label{sec2}
\noindent{\it Proof of Theorem \ref{the3}.}

As described in Section \ref{sec4}, we compute
\begin{equation*}%\label{e32b}
\int_{0}^{R} da_1 \, a_1 \int_{0}^{\infty} da_2 \, a_2 \int_{0}^{\alpha} d\alpha_1 \int_{0}^{\alpha} d\alpha_2
\, \hat{p}_{W_{2\pi}}(A_1,A_2;s),
\end{equation*}
where
 \begin{align*}%\label{e32}
\hat{p}_{W_{2\pi}}(A_1,A_2;s)&=\frac{1}{\pi^2}\int_0^{\infty}d\theta K_{i\theta}(\sqrt s a_1)K_{i\theta}(\sqrt s a_2)\nonumber \\ &\times\left(\cosh((\pi-|\alpha_1-\alpha_2|)\theta)-\frac{\sinh(\pi\theta)}{\sinh(2\pi \theta)}\left(\cosh((2\pi-\alpha_1-\alpha_2)\theta)+\cosh((\alpha_1-\alpha_2)\theta)\right)\right),
\end{align*}
and then take the inverse Laplace transform.
A straightforward computation shows
\begin{align}\label{e33}
\int_0^{\alpha}d\alpha_1\int_0^{\alpha}d\alpha_2&\bigg(\cosh((\pi-|\alpha_1-\alpha_2|)\theta)-\frac{\sinh(\pi\theta)}{\sinh(2\pi \theta)}\bigg(\cosh((2\pi-\alpha_1-\alpha_2)\theta)+\cosh((\alpha_1-\alpha_2)\theta)\bigg)\bigg)\nonumber \\ &
=\frac{2\alpha}{\theta}\sinh(\pi\theta)+\frac{1}{2\theta^2\cosh(\pi\theta)}\big(3-3\cosh(2\pi\theta)\big)\nonumber \\ &\ \ \ +
\frac{1}{2\theta^2\cosh(\pi\theta)}\big(4\cosh((2\pi-\alpha)\theta)-\cosh((2\pi-2\alpha)\theta)-3\big)\nonumber \\ &:=C_1+C_2+C_3,
\end{align}
with obvious notation.

We obtain by definition of $C_1$, Fubini's theorem, \eqref{e34}
and \eqref{e35},
\begin{align*}%\label{e36}
\int_0^Rda_1&\, a_1\int_0^{\infty}da_2\, a_2\frac{1}{\pi^2}\int_0^{\infty}d\theta K_{i\theta}(\sqrt s a_1)K_{i\theta}(\sqrt s a_2)C_1\nonumber \\ &
=\frac{2\alpha }{\pi s}\int_0^Rda_1\, a_1\int_0^{\infty}d\theta\, K_{i\theta}(\sqrt sa_1)\cosh(\pi \theta/2)\nonumber \\ &
=\frac{\alpha R^2}{2s}.
\end{align*}

So $L^{-1}\big \{
(2s)^{-1}\alpha R^2\big \}
(t)= 2^{-1}\alpha R^2$, which is the first term in right-hand side of \eqref{e20}.

Furthermore, by Fubini's theorem, \eqref{e34}, and the definition of $C_2$ in \eqref{e33}, we find
\begin{align}\label{e37}
\int_0^Rda_1&\, a_1\int_0^{\infty}da_2\, a_2\frac{1}{\pi^2}\int_0^{\infty}d\theta K_{i\theta}(\sqrt s a_1)K_{i\theta}(\sqrt s a_2)C_2\nonumber \\ &
=-\frac{3}{\pi^2}\int_0^Rda_1\, a_1\int_0^{\infty}da_2\, a_2\int_0^{\infty}\frac{d\theta}{\theta^2} K_{i\theta}(\sqrt s a_1)K_{i\theta}(\sqrt s a_2)\tanh(\pi\theta) \sinh(\pi\theta)
\nonumber \\ &
=-\frac{3}{\pi}\int_0^Rda_1\,a_1\int_0^{\infty}\frac{d\theta}{\theta s}K_{i\theta}(\sqrt s a_1)
\tanh(\pi\theta)\cosh(\pi\theta/2)
 \nonumber \\ &
=-\frac{3}{\pi}\int_0^Rda_1\,a_1\int_0^{\infty}\frac{d\theta}{\theta s}K_{i\theta}(\sqrt s a_1)\bigg(\sinh(\pi\theta/2)+\frac{\sinh(\pi\theta/2)}{\cosh(\pi\theta)}\bigg).
\end{align}
See also (2.9) in \cite{MvdBSS2}.
By \eqref{e35a}, and Fubini's theorem (see (2.10) in \cite{MvdBSS2}),
\begin{align}\label{e38}
-\frac{3}{\pi s}\int_0^{\infty}\frac{d\theta}{\theta}K_{i\theta}(\sqrt s a_1)\sinh(\pi\theta/2)&=-\frac{3}{\pi s}\int_0^{\pi/2}d\eta \int_{0}^{\infty} d\theta\, \cosh(\eta\theta)K_{i\theta}(\sqrt s a_1)\nonumber \\ &=-\frac{3}{2s}\int_0^{\pi/2}d\eta e^{-a_1\sqrt s\cos\eta}.
\end{align}
By \eqref{e38}, we obtain for the first term in the right-hand side of \eqref{e37}
\begin{align}\label{e39}
-\frac{3}{\pi}&\int_0^Rda_1\,a_1\int_0^{\infty}\frac{d\theta}{\theta s}K_{i\theta}(\sqrt s a_1)\sinh(\pi\theta/2)\nonumber \\ &=-\frac{3}{2s}\int_0^Rda\, a\int_0^{\pi/2}d\eta e^{-a\sqrt s\cos\eta}.
\end{align}
From the calculation in (2.14) of \cite{MvdBSS2}, we find that the inverse Laplace transform of the right-hand side of \eqref{e39} is given by
\begin{equation*}%\label{e40}
-\frac{3Rt^{1/2}}{\pi^{1/2}}+\frac{3Rt^{1/2}}{\pi^{1/2}}\int_1^{\infty}\frac{dv}{v^2}\int_0^1\frac{\zeta d\zeta}{(1-\zeta^2)^{1/2}}e^{-R^2\zeta^2v^2/(4t)}.
\end{equation*}
For the second term in the right-hand side of \eqref{e37}, by Fubini's theorem and \eqref{e34}, we have
\begin{align}\label{e41}
-\frac{3}{\pi}&\int_0^Rda_1\,a_1\int_0^{\infty}\frac{d\theta}{\theta s}K_{i\theta}(\sqrt s a_1)\frac{\sinh(\pi\theta/2)}{\cosh(\pi\theta)}\nonumber \\ &=
-\frac{3}{2s^2
 }\int_0^{\infty}\frac{d\theta}{\cosh(\pi\theta)}+\frac{3}{\pi}\int_R^{\infty}da_1\,a_1\int_0^{\infty}\frac{d\theta}{\theta s}K_{i\theta}(\sqrt s a_1)\frac{\sinh(\pi\theta/2)}{\cosh(\pi\theta)}\nonumber \\ &=-\frac{3}{4s^2}+\frac{3}{\pi}\int_R^{\infty}da_1\,a_1\int_0^{\infty}\frac{d\theta}{\theta s}K_{i\theta}(\sqrt s a_1)\frac{\sinh(\pi\theta/2)}{\cosh(\pi\theta)}.
\end{align}
Taking the inverse Laplace transform of the first term in the right-hand side of \eqref{e41} yields $-\frac34 t$, which accounts for the $-\frac34$ term in \eqref{e19}.

By Fubini's theorem and \eqref{e42}, we obtain
\begin{align*}%\label{e43}
\frac{3}{\pi}\int_R^{\infty}da\,a&\int_0^{\infty}\frac{d\theta}{\theta s}K_{i\theta}(\sqrt s a)\frac{\sinh(\pi\theta/2)}{\cosh(\pi\theta)}\nonumber \\ &=\frac{3}{\pi}\int_R^{\infty}da\,a\int_0^{\infty} dw e^{-\sqrt s a\cosh w}\int_0^{\infty}\frac{d\theta}{\theta s} \cos(w\theta) \frac{\sinh(\pi\theta/2)}{\cosh(\pi\theta)}.
\end{align*}
By \eqref{e35d},
\begin{equation}\label{e44}
L^{-1}\big\{
s^{-1}e^{-\sqrt s a\cosh w}\big\}(t)=\frac{2}{\sqrt \pi}\int_{(a\cosh w)/(4t)^{1/2}}^{\infty}dr\,e^{-r^2}.
\end{equation}
Hence the inverse Laplace transform of the second term in the right-hand side of \eqref{e41}
is bounded in absolute value by
\begin{align}\label{e45}
\frac{3}{\pi}\bigg|\int_R^{\infty}da&\,a\int_0^{\infty} dw \textup{Erfc}\,((a\cosh w)/(4t)^{1/2})\int_0^{\infty}\frac{d\theta}{\theta} \cos(w\theta)\frac{\sinh(\pi\theta/2)}{\cosh(\pi\theta)}\bigg|\nonumber \\ & \le \frac{3}{\pi}\int_R^{\infty}da\,a\int_0^{\infty} dw \textup{Erfc}\,((a\cosh w)/(4t)^{1/2})\int_0^{\infty}\frac{d\theta}{\theta}\frac{\sinh(\pi\theta/2)}{\cosh(\pi\theta)}\nonumber \\ &=\frac{3\log(1+\sqrt 2)}{\pi}\int_0^{\infty} dw \int_R^{\infty}da\,a\textup{Erfc}\,((a\cosh w)/(4t)^{1/2}),
\end{align}
where we have used \eqref{e35b}. Since for $z\ge0$,
\begin{equation}\label{e46}
\textup{Erfc}(z):=\frac{2}{\sqrt \pi}\int_z^{\infty}dr\,e^{-r^2}= \frac{2}{\sqrt \pi}\int_0^{\infty}dr\,e^{-(z+r)^2}\le \frac{2}{\sqrt \pi}\int_0^{\infty}dr\,e^{-z^2-r^2}=e^{-z^2},
\end{equation}
we obtain that the right-hand side of \eqref{e45} is bounded from above by
\begin{align}\label{e47}
\frac{3\log(1+\sqrt 2)}{\pi}\int_0^{\infty} dw\, \int_R^{\infty}da\,a\,e^{-(a\cosh w)^2/(4t)}&=\frac{6t\log(1+\sqrt 2)}{\pi}\int_0^{\infty}
\frac{dw}{(\cosh w)^2}e^{-(R\cosh w)^2/(4t)}\nonumber \\ &=O(te^{-R^2/(4t)}).
\end{align}

In order to compute $C_3$, we extend the integral with respect to $a_1$ to the interval $[0,\infty)$, and obtain, via Fubini's theorem and \eqref{e34},
\begin{align*}%\label{e48}
\int_0^{\infty}da_1&\, a_1\int_0^{\infty}da_2\, a_2\frac{1}{\pi^2}\int_0^{\infty}d\theta K_{i\theta}(\sqrt s a_1)K_{i\theta}(\sqrt s a_2)C_3\nonumber \\ &=
\int_0^{\infty}da_1\, a_1\int_0^{\infty}da_2\, a_2\frac{1}{\pi^2}\int_0^{\infty}d\theta K_{i\theta}(\sqrt s a_1)K_{i\theta}(\sqrt s a_2)\nonumber \\ &\,\,\,\, \,\,\,\,\times\frac{1}{2\theta^2\cosh(\pi\theta)}\big(4\cosh((2\pi-\alpha)\theta)-\cosh((2\pi-2\alpha)\theta)-3\big)\nonumber \\ &=
\frac{1}{8s^2}\int_0^{\infty}d\theta
\, \frac{4\cosh((2\pi-\alpha)\theta)-\cosh((2\pi-2\alpha)\theta)-3}{\cosh(\pi\theta)(\sinh(\pi\theta/2))^2}\nonumber \\ &=\frac{1}{4s^2}\int_0^{\infty}d\theta\, \frac{4(\sinh((\pi-\frac{\alpha}{2})\theta))^2-(\sinh((\pi-\alpha)\theta))^2}{ \big(\sinh(\pi\theta/2)\big)^2 \cosh(\pi\theta)}.
\end{align*}
Inverting the Laplace transform yields a contribution $\big(\frac34+a(\alpha)\big)t,$
where $a(\alpha)$ is as defined in \eqref{e19}. This, together with the statement below \eqref{e41}
gives the contribution $a(\alpha)t$ in \eqref{e20}.

It remains to bound the inverse Laplace transform of
\begin{align}\label{e49}
\int_R^{\infty}da_1&\, a_1\int_0^{\infty}da_2\, a_2\frac{1}{\pi^2}\int_0^{\infty}d\theta K_{i\theta}(\sqrt s a_1)K_{i\theta}(\sqrt s a_2)\nonumber \\ &\,\,\,\,\times\frac{1}{2\theta^2\cosh(\pi\theta)}\big(4\cosh((2\pi-\alpha)\theta)-\cosh((2\pi-2\alpha)\theta)-3\big)\nonumber \\ &
=\frac{1}{ s}\int_R^{\infty}da\, a\int_0^{\infty}d\theta\, K_{i\theta}(\sqrt s a)\frac{4\cosh((2\pi-\alpha)\theta)-\cosh((2\pi-2\alpha)\theta)-3}{4\pi\theta\sinh(\pi\theta/2)\cosh(\pi\theta)}.
\end{align}
We first consider the case $\pi/2<\alpha<7\pi/4$, and we proceed as above. We use \eqref{e42}, and invert the Laplace transform of $s^{-1}e^{-\sqrt s a\cosh w}$ as in \eqref{e44}. This gives that the inverse Laplace transform of \eqref{e49} equals
\begin{align}\label{e50}
\int_R^{\infty}da\,a\int_0^{\infty} dw\, \textup{Erfc}\,((a\cosh w)/(4t)^{1/2})\int_0^{\infty}d\theta\,(\cos(w\theta))\frac{4\cosh((2\pi-\alpha)\theta)
-\cosh((2\pi-2\alpha)\theta)-3}{4\pi\theta\sinh(\pi\theta/2)\cosh(\pi\theta)}.
\end{align}
Using $|\cos(w\theta)|\le 1$, we find that the absolute value of the expression under \eqref{e50} is bounded from above by
\begin{align}\label{e51}
\int_R^{\infty}da\,a\int_0^{\infty} dw\, \textup{Erfc}\,((a\cosh w)/(4t)^{1/2})\int_0^{\infty}&d\theta\,\frac{4\cosh((2\pi-\alpha)\theta)-\cosh((2\pi-2\alpha)\theta)-3}{4\pi\theta\sinh(\pi\theta/2)\cosh(\pi\theta)}\nonumber \\ &=O(te^{-R^2/(4t)}),
\end{align}
where, as before, we have used \eqref{e46}, and argued similarly to \eqref{e47}.
We note that the integrals with respect to $\theta$ in \eqref{e50} and \eqref{e51} converge for $\pi/2<\alpha<7\pi/4$.

We next consider the case $7\pi/4<\alpha<2\pi$. We write the right-hand side of \eqref{e49}
as the sum of two terms, say $D_1(s)+D_2(s)$, where
\begin{align}\label{e52}
D_1(s)=\frac{1}{ s}\int_R^{\infty}da\, a\int_0^{\infty}d\theta K_{i\theta}(\sqrt s a)\frac{\cosh((2\pi-\alpha)\theta)-1}{\pi\theta\sinh(\pi\theta/2)\cosh(\pi\theta)},
\end{align}
and
\begin{equation}\label{e53}
D_2(s)=\frac{1}{s}\int_R^{\infty}da\, a\int_0^{\infty}d\theta K_{i\theta}(\sqrt s a)\frac{1-\cosh((2\pi-2\alpha)\theta)}{4\pi\theta\sinh(\pi\theta/2)\cosh(\pi\theta)}.
\end{equation}
Using \eqref{e42}, \eqref{e44} gives
\begin{align}\label{e54}
\big|L^{-1}\{D_1\}
(t) \big|&=
\bigg|\int_R^{\infty}da\,a\int_0^{\infty} dw\, \textup{Erfc}\,((a\cosh w)/(4t)^{1/2})\int_0^{\infty}d\theta\, \cos(w\theta)\frac{\cosh((2\pi-\alpha)\theta)-1}{\pi\theta\sinh(\pi\theta/2)\cosh(\pi\theta)}\bigg|\nonumber \\ &\le
\int_R^{\infty}da\,a\int_0^{\infty} dw\, \textup{Erfc}\,((a\cosh w)/(4t)^{1/2})\int_0^{\infty}d\theta\,\frac{\cosh((2\pi-\alpha)\theta)-1}{\pi\theta\sinh(\pi\theta/2)\cosh(\pi\theta)}\nonumber \\ &=O(te^{-R^2/(4t)}),
\end{align}
where we have used \eqref{e46}, and argued similarly to \eqref{e47}. The integral with respect to $\theta$ in \eqref{e54} converges for $\alpha\in(\pi/2,2\pi)  \supset
(7\pi/4,2\pi)$.
To invert $D_2(s)$ we rewrite the integrand as follows. For $\epsilon\in \R$,
\begin{align*}%\label{e55}
&\frac{1-\cosh((2\pi-2\alpha)\theta)}{4\pi\theta\sinh(\pi\theta/2)\cosh(\pi\theta)}\nonumber \\ &
=\frac{4\cosh(\pi\theta/2)-2\cosh((2\alpha-\frac{3\pi}{2})\theta)-2\cosh((2\alpha-\frac{5\pi}{2})\theta)}{4\pi\theta\sinh(2\pi\theta)}\nonumber \\ &
=\frac{4\cosh(\pi\theta/2)-2\cosh((2\alpha-\frac{3\pi}{2})\theta)-2\cosh((2\alpha-
\frac{5\pi}{2})\theta)+2\cosh((2\pi+\epsilon)\theta)-2\cosh((2\pi-\epsilon)\theta)}{4\pi\theta\sinh(2\pi\theta)}\nonumber \\ &
-\frac{1}{\pi\theta}\sinh(\epsilon\theta).
\end{align*}
We choose $2\alpha-\frac{3\pi}{2}=2\pi+\epsilon$. This gives that $\epsilon=2\alpha-\frac{7\pi}{2}$, and
\begin{align}\label{e56}
\frac{1-\cosh((2\pi-2\alpha)\theta)}{4\pi\theta\sinh(\pi\theta/2)\cosh(\pi\theta)}&=\frac{2\cosh(\pi\theta/2)-\cosh((2\alpha-
\frac{5\pi}{2})\theta)-\cosh((\frac{11\pi}{2}-2\alpha)\theta)}{2\pi\theta\sinh(2\pi\theta)}\nonumber \\ &
\,\,\,\,-\frac{1}{\pi\theta}\sinh((4\alpha-7\pi)\theta/2).
\end{align}
The first term in the right-hand side of \eqref{e56} is integrable, and, analogously to the above, we proceed with \eqref{e42}, \eqref{e44}, and \eqref{e46}. This gives a remainder
$O(te^{-R^2/(4t)}).$

It remains to invert the contribution coming from the second term in the right-hand side of \eqref{e56}. We recall (2.18) in \cite{MvdBSS2}. That is, for $-\frac{\pi}{2}<\beta<\frac{\pi}{2}$, by Fubini's theorem, \eqref{e35a}, and \eqref{e35d}, we have
\begin{align}\label{e57}
L^{-1}\bigg\{\int_R^{\infty}da\, a\int_0^{\infty}\frac{d\theta}{\theta s}K_{i\theta}(\sqrt s a)\sinh(\beta\theta)\bigg\}(t)
&=L^{-1}\bigg\{\int_R^{\infty}da\, \frac{a}{s}\int_0^{\beta}d\eta\,\int_0^{\infty}d\theta \cosh(\eta \theta)K_{i\theta}(\sqrt s a)\bigg\}(t)\nonumber \\ &
=L^{-1}\bigg\{\frac{\pi}{2s}\int_R^{\infty}da\, a\int_0^{\beta}d\eta\,e^{-a\sqrt s \cos\eta}\bigg\}(t)\nonumber \\ &=\frac{\pi}{2}\int_R^{\infty}da\,a\int_0^{\beta}d\eta\, \textup{Erfc}\bigg(\frac{a\cos \eta}{(4t)^{1/2}}\bigg)\nonumber \\ &\le  \frac{\pi\beta}{2}\int_R^{\infty}da\, a\textup{Erfc}\bigg(\frac{a\cos \beta}{(4t)^{1/2}}\bigg)\nonumber\\ &=O(te^{-R^2(\cos \beta)^2/(4t)}),
\end{align}
where we have used once more \eqref{e46}.

For $7\pi/4<\alpha<2\pi$ we have that $2\alpha-\frac{7\pi}{2}\in (0,\pi/2)$. Hence the second term in the right-hand side of \eqref{e56} gives a contribution $O(te^{-R^2(\sin(2\alpha))^2/(4t)})$.

For $\alpha=\frac{7\pi}{4}$, we have
\begin{equation*}%\label{e58}
\int_0^{\infty} d\theta\,\bigg|\frac{\cosh((2\pi-\alpha)\theta)-1}{\pi\theta\sinh(\pi\theta/2)\cosh(\pi\theta)}\bigg|<\infty.
\end{equation*}
Hence the inverse Laplace transform of $D_1$ is $O(t e^{-R^2/(4t)})$.
For $\alpha=7\pi/4$ we rewrite \eqref{e53} as
\begin{align*}%\label{e59}
D_2&=\frac{1}{s}\int_R^{\infty}da\, a\int_0^{\infty}d\theta K_{i\theta}(\sqrt s a)\frac{\cosh(\pi\theta/2)-\frac12\cosh(\pi\theta)-\frac12\cosh(2\pi\theta)}{\pi\theta\sinh(2\pi\theta)}\nonumber \\ &
=\frac{1}{s}\int_R^{\infty}da\, a\int_0^{\infty}d\theta K_{i\theta}(\sqrt s a)\bigg\{\frac{\cosh(\pi\theta/2)-\frac12\cosh(\pi\theta)-\frac12
 }{\pi\theta\sinh(2\pi\theta)}-\frac{\tanh(\pi\theta)}{2\pi\theta}\bigg\}.
\end{align*}
Since
\begin{equation*}%\label{e60}
\int_0^{\infty} d \theta \bigg|\frac{\cosh(\pi\theta/2)-\frac12\cosh(\pi\theta)-\frac12}{\pi\theta\sinh(2\pi\theta)}\bigg|<\infty,
\end{equation*}
we have that this part also gives
a contribution  $O(te^{-R^2/(4t)})$. By \eqref{e35c} and \eqref{e35d} (see (2.20) in \cite{MvdBSS2}), we have for $\beta>0$,
\begin{align}\label{e61}
L^{-1}\bigg\{\frac{1}{\pi s}&\int_R^{\infty}da\, a\int_0^{\infty}\frac{d\theta}{\theta}\tanh(\beta\theta) K_{i\theta}(\sqrt s a)\bigg\}(t)\nonumber \\ &=
L^{-1}\bigg\{\frac{1}{\pi s}\int_R^{\infty}da\, a\int_0^{\infty}dw\, e^{-a\sqrt s\cosh w}\int_0^{\infty}\frac{d\theta}{\theta}\tanh(\beta\theta)\cos(w\theta) \bigg\}(t)\nonumber \\ &=L^{-1}\bigg\{\frac{1}{\pi s}\int_R^{\infty}da\, a\int_0^{\infty}dw\, e^{-a\sqrt s\cosh w}\log(\coth(\pi w/(4\beta)))\bigg\}(t) \nonumber \\ &
=\frac{1}{\pi}\int_R^{\infty}da\, a \int_0^{\infty}dw\,\textup{Erfc}\bigg(\frac{a\cosh w}{(4t)^{1/2}}\bigg)\log(\coth(\pi w/(4\beta))).
\end{align}
By \eqref{e46} we obtain that \eqref{e61} is bounded from above by
\begin{equation*}%\label{e62}
2\pi^{-1}te^{-R^2/(4t)}\int_0^{\infty}  dw \, \frac{\log(\coth(\pi w/(4\beta)))}{(\cosh w)^2}=O(te^{-R^2/(4t)}).
\end{equation*}
In particular, for $\beta = \pi$, the term $-\frac{\tanh(\pi\theta)}{2\pi\theta}$ contributes
a remainder $O(te^{-R^2/(4t)})$.

For $\alpha=\pi/2$, we have
\begin{equation*}%\label{e63}
\int_0^{\infty}\bigg|\frac{1-\cosh((2\pi-\alpha)\theta)}{4\pi\theta\sinh(\pi\theta/2)\cosh(\pi\theta)}\bigg|<\infty.
\end{equation*}
Hence the inverse Laplace transform of $D_2$ is, for $\alpha=\pi/2$, $O(te^{-R^2/(4t)})$. On the other hand, for $\alpha=\pi/2$ the integrand in \eqref{e52} equals the integrand of \eqref{e53} for $\alpha=7\pi/4$ up to a
factor of $-\frac14$. Hence the inverse Laplace transform  of $D_1$ is also $O(te^{-R^2/(4t)})$.

For $\pi/4<\alpha<\pi/2$, we have
\begin{equation*}%\label{e64}
\int_0^{\infty}\bigg|\frac{1-\cosh((2\pi-2\alpha)\theta)}{4\pi\theta\sinh(\pi\theta/2)\cosh(\pi\theta)}\bigg|<\infty.
\end{equation*}
Hence the inverse Laplace transform of $D_2$ is, for $\pi/4<\alpha<\pi/2$, $O(te^{-R^2/(4t)})$.
Similarly to the above, we rewrite the hyperbolic part of the integrand in \eqref{e52} as follows:
\begin{align}\label{e65}
&\frac{\cosh((2\pi-\alpha)\theta)-1}{\pi\theta\sinh(\pi\theta/2)\cosh(\pi\theta)}\nonumber \\ &
\,\,\,\,=\frac{2\cosh((\frac{5\pi}{2}-\alpha)\theta)+2\cosh((\frac{3\pi}{2}-\alpha)\theta)-4\cosh(\pi\theta/2)}{\pi\theta\sinh(2\pi\theta)}\nonumber \\ &
\,\,\,\,=\frac{2\cosh((\frac{5\pi}{2}-\alpha)\theta)+2\cosh((\frac{3\pi}{2}-\alpha)\theta)-4\cosh(\pi\theta/2)-2\cosh((2\pi+\epsilon)\theta)
+2\cosh((2\pi-\epsilon)\theta)}{ \pi\theta\sinh(2\pi\theta)}\nonumber \\ &
\hspace{5mm}+\frac{ 4\sinh(\epsilon\theta)}{ \pi\theta}.
\end{align}
We subsequently choose $\epsilon=\frac{\pi}{2}-\alpha$. With this choice of $\epsilon$, the absolute value of the first term in the right-hand side of \eqref{e65}
is integrable with respect to $\theta$ on $\R^+$. Hence this term contributes $O(te^{-R^2/(4t)})$ to the inverse Laplace transform of the corresponding integral in \eqref{e52}. Moreover, since $\epsilon\in (0,\pi/2)$ for this case, we have by \eqref{e57} that this term contributes $O(te^{-R^2\sin^2(\alpha)/(4t)})$ to the inverse Laplace transform of the corresponding integral in \eqref{e52}.

We next consider the case $\alpha=\pi/4$. Then $D_2$ for $\pi/4$ equals $D_2$ for $7\pi/4$, we immediately conclude that this term is $O(te^{-R^2/(4t)})$.
We rewrite the hyperbolic part of the integrand as follows:
\begin{align}\label{e66}
&\frac{\cosh(7\pi\theta/4)-1}{\pi\theta\sinh(\pi\theta/2)\cosh(\pi\theta)}\nonumber \\ &\,\,=
\frac{2\cosh(9\pi\theta/4)+2\cosh(5\pi\theta/4)-2\cosh((2\pi+\epsilon)\theta)+2\cosh((2\pi-\epsilon)\theta)-4\cosh(\pi\theta/2)}{\pi\theta\sinh(2\pi\theta)}
+\frac{4\sinh(\epsilon\theta)}{\pi\theta}.
\end{align}
We subsequently choose $\epsilon=\frac{\pi}{4}$ and observe that the absolute value of the first term in the right-hand side of \eqref{e66} is integrable.
This then yields that the corresponding Laplace transform is $O(te^{-R^2/(4t)})$.
The second term has been inverted in \eqref{e57}. Choosing $\beta=\pi/4$ gives a remainder
$O(te^{-R^2/(8t)})$.

We finally consider the case $0<\alpha<\pi/4$. The contribution from $D_1$ to the inverse Laplace transform can be estimated by \eqref{e65}, and the lines below, since
$\epsilon\in (0,\pi/2)$ for this case too. Hence we obtain a remainder $O(te^{-R^2(\sin\alpha)^2/(4t)})$. The contribution from $D_2$ to the inverse Laplace transform follows by a minor modification of \eqref{e65}. We  have
\begin{align}\label{e67}
&\frac{1-\cosh((2\pi-2\alpha)\theta)}{4\pi\theta\sinh(\pi\theta/2)\cosh(\pi\theta)}\nonumber \\ &
\,\,\,\,=\frac{\cosh(\pi\theta/2)-\frac12\cosh((\frac{5\pi}{2}-2\alpha)\theta)-\frac12\cosh((\frac{3\pi}{2}-2\alpha)\theta)+
\frac12\cosh((2\pi+\epsilon)\theta)-\frac12\cosh((2\pi-\epsilon)\theta)}{\pi\theta\sinh(2\pi\theta)}\nonumber \\& \hspace{5mm} -\frac{\sinh(\epsilon\theta)}{\pi\theta}.
\end{align}
We choose $\epsilon=\frac{\pi}{2}-2\alpha\in (0,\pi/2)$, and obtain the remainder $O(te^{-R^2\sin^2(2\alpha)/(4t)})$ from the corresponding integral in \eqref{e53} by \eqref{e57}. The first term in the right-hand side of \eqref{e67} gives $O(te^{-R^2/( 4t )})$.
\hspace*{\fill }$\square $
%%%%%%%%%%%%%%%%%%%%%%%%%%%%%%
\bigskip
%\ack{acknowledgment}

%%%%%%%%%%%%%%%%%%%%%%%%%%%%%%


\begin{thebibliography}{999}

\bibitem{MvdB1} M.\ van den Berg,  \emph {Heat flow and perimeter in $\mathbb{R}^m$}, Potential Analysis {\bf39}, 369--387 (2013).

\bibitem{MvdBEBD}M.\ van den Berg, E.\ B.\ Davies, \emph {Heat flow out of regions
in ${\R}^{m}$}, Mathematische Zeitschrift {\bf 202}, 463--482
(1989).


\bibitem{MvdBPG}
M.\ van den Berg, P.\ B.\ Gilkey,
\emph{Heat flow out of a compact manifold}, Journal of Geometric Analysis, \textbf{25}, 1576--1601 (2015).


\bibitem{MvdBGKK}M.\ van den Berg, P.\ Gilkey, K.\ Kirsten, V.\  A.\ Kozlov,
\emph {Heat content asymptotics for Riemannian manifolds with
Zaremba boundary conditions}, Potential Analysis \textbf{26},
225--254 (2007).

\bibitem{MvdBKG1} M.\ van den Berg, K.\ Gittins, \emph {Uniform bounds for the heat content of open sets in Euclidean space},
Differential Geometry and its Applications \textbf{40}, 67--85 (2015).

\bibitem{MvdBKG2} M.\  van den Berg, K.\ Gittins, \emph {On the heat content of a polygon}, Journal of Geometric
Analysis, Journal of Geometric Analysis \textbf{26}, 2231--2264 (2016).

\bibitem{MvdBSS1} M.\ van den Berg, S.\ Srisatkunarajah, \emph {Heat equation for
a region in ${\R}^{2}$ with a polygonal boundary}, Journal of
the London Mathematical Society (2) {\bf 37}, 119--127 (1988).

\bibitem{MvdBSS2} M.\ van den Berg, S.\ Srisatkunarajah, \emph {Heat flow and Brownian motion for a region in
${\R}^{2}$ with a polygonal boundary}, Probability Theory and
Related Fields {\bf 86}, 41--52 (1990).


\bibitem{CarslawJ}H.\ S.\ Carslaw, J.\ C.\ Jaeger,
Conduction of Heat in Solids. Clarendon Press, Oxford
(2000).

\bibitem{CJ} J.\ C.\ Cooke, \emph {Note on a heat conduction problem}, Amer. Math. Monthly \textbf{62}, 331--334 (1955).

\bibitem{E} A.\ Erd\'elyi, Tables of Integral Transforms I. McGraw-Hill, New York (1954).

\bibitem{AG}A.\ Grigor'yan, Heat kernel and analysis on manifolds. AMS/IP Studies in Advanced Mathematics, \textbf{47}. American Mathematical Society, Providence, RI; International Press, Boston, MA, (2009).

\bibitem{PG}P.\  Gilkey, Asymptotic Formulae in Spectral
Geometry, Stud. Adv. Math., Chapman \& Hall/CRC, Boca Raton, FL
(2004).

\bibitem{GR} I.\ S.\ Gradshteyn, I.\ M.\ Ryzhik, Table of integrals, series, and products. Elsevier/Academic Press, Amsterdam (2015).

\bibitem{KM} M.\ Kac, \emph {Can one hear the shape of a drum ?} Amer. Math. Monthly \textbf{73}, 1--23 (1966).

\bibitem{MKS} H.\ P.\ McKean, I.\ M.\ Singer \emph {Curvature and the eigenvalues of the Laplacian}, J. Differential Geometry \textbf{1}, 43--69 (1967).


\bibitem{NRS} M.\ Nursultanov, J.\ Rowlett, D.\ A.\ Sher, \emph{The heat kernel on curvilinear polygonal domains in surfaces}. arXiv:1905.00259 [math.AP] (1 May 2019).


\bibitem{SF} F.\ Spitzer, \emph {Some theorems concerning 2-dimensional Brownian motion}. Trans. Amer. Math. Soc. \textbf{87}, 187--197 (1958).







\end{thebibliography}
\end{document}